%% file: admissibility.tex
\documentclass[a4paper,11pt, final]{amsart}
\usepackage[margin=7cc]{geometry}
\usepackage[textsize=footnotesize,color=yellow!50]{todonotes}

\usepackage[utf8]{inputenc}
\usepackage[T1]{fontenc}
\usepackage[english]{babel}
\usepackage{csquotes}
\usepackage{amssymb}
\usepackage{amsmath,amsthm}
\usepackage[notref,notcite]{showkeys}
\usepackage{datetime}

\usepackage[shortlabels]{enumitem}
\usepackage[pdfauthor={Duanmu, Roy, Schrittesser},
            pdftitle={Admissibility is Bayes optimality with infinitesimals},
            pdfproducer={Latex with hyperref}]{hyperref}

\include{preamble}

\author[Duanmu]{Haosui Duanmu}
\address{Haosui Duanmu, 
Harbin Institute of Technology, China
}
\email{duanmuhaosui@hotmail.com}

\author[Roy]{Daniel M. Roy}
\address{Daniel M. Roy, University of Toronto, Toronto, Canada}
\email{daniel.roy@utoronto.ca}

\author[Schrittesser]{David Schrittesser}
\address{David Schrittesser, University of Toronto, Toronto, Canada}
\email{david@logic.univie.ac.at}

\title%
{Admissibility is Bayes optimality with infinitesimals}



\begin{document}

\begin{abstract}
We give an exact characterization of admissibility in statistical decision problems 
in terms of Bayes optimality in a so-called nonstandard extension of the original decision problem, 
as introduced by Duanmu and Roy.
Unlike the consideration of improper priors or other generalized notions of Bayes optimalitiy,
the nonstandard extension is distinguished, in part, by having priors that can assign ``infinitesimal'' mass in a sense that can be made rigorous using results from nonstandard analysis.
With these additional priors, we 
find that, informally speaking,
a decision procedure $\delta_0$ is admissible in the original statistical decision problem
if and only if, in the nonstandard extension of the problem, 
the nonstandard extension of $\delta_0$ is Bayes optimal among the extensions of standard decision procedures
with respect to a nonstandard prior 
that assigns at least infinitesimal mass to every standard parameter value.
We use the above theorem to give further characterizations of admissibility, one related to Blyth's method,  one to 
a condition due to Stein which characterizes admissibility under some regularity assumptions; and finally, a characterization using finitely additive priors in decision problems meeting certain regularity requirements.
Our results imply that Blyth's method is a sound and complete method for establishing admissibility. 
Buoyed by this result, we revisit the univariate two-sample common-mean problem, and 
show that the Graybill--Deal estimator is admissible among a certain class of unbiased decision procedures.
\end{abstract}

\maketitle

\tableofcontents

\section{Introduction}

In his work introducing statistical decision theory,
Wald formalized various notions of optimality in
terms of partial and total orders on the class of  all decision procedures.
An admittedly weak optimality criterion was captured by a strict partial order on decision procedures, wherein one decision procedure \emph{dominates} another if its 
risk is no larger at \emph{every} parameter value and strictly smaller at \emph{some} parameter value.
The minimal elements---those not dominated by any other procedure---were deemed to be \emph{admissible} and a subclass of procedures was called complete if every procedure not in the class was dominated by a procedure within it.  
Wald also defined total orders in terms of the worst-case (minimax) risk and average case (Bayes) risk, where the average is taken with respect to some ``prior'', i.e., probability measure on the set of parameters.

Some of the key results from statistical decision theory pertain to relationships between these various notions of optimality
and prescriptions for choosing estimators in terms of characterizations of complete classes.
Among Wald's numerous results, the most paradigmatic are perhaps his complete class theorems 
establishing the completeness of the class of Bayes decision procedures, i.e., the minimal elements with respect to Bayes risk for \emph{some} prior probability measure. See, e.g., Theorems~4.11 and 4.14 in \cite{wald1947} for early examples.
By these theorem, in decision problems satisfying strong assumptions regarding compactness and continuity of various of its elements, 
every admissible procedure is Bayes. 
Examples show this implication does not hold without strong assumptions. 
Indeed, once you drop compactness or continuity, one can find problems with admissible procedures that are not Bayes.
In order to reestablish a link between admissibility and Bayes optimality in more general problems,
generalizations of Bayes optimality were introduced by numerous authors, 
based on ``improper'' priors, sequences of priors, or limits of sequences of Bayes procedures.
In each case, however, strong regularity conditions are exploited, raising the prospect that the relationship between admissibility and Bayes optimality is somewhat fragile.

Despite a frenzy of work in the decades following Wald's seminal work and isolated breakthroughs over the past seventy years,
our understanding of the relationship between admissibility and Bayes optimality is still incomplete.
We should, however, be uncomfortable with our incomplete understanding,
especially since this long standing problem seems to have been abandoned mostly out of exhaustion.
While considerations of minimaxity or minimax rates tend to dominate contemporary statistical theory,
there is more to learn about admissibility and Bayes optimality, and new understanding may 
cast light on new problems, such as understanding adaptivity.

\medskip

While a complete class theorem may imply that one can restrict attention to a certain class of Bayes procedures,
a Bayes procedure may still be inadmissible.
Sufficient conditions for admissibility in terms of Bayes optimality are, however, the principal tools used to establish admissibility. 
The essential lack of matching necessary conditions means that we do not understand the weaknesses of these tools.
In practice, establishing the admissibility of estimators is often nontrivial.
Indeed, for some thoroughly studied estimators, the question of their admissibility remains open. 
A well-known example is the so-called Graybill--Deal estimator of the common mean of multiple independent normal samples. (We study this estimator in Section~\ref{s.GD}.)  

In any decision problem, it is easy to show that a procedure is admissible if it is the unique Bayes procedure with respect to some prior is admissible.
As examples show, this implication is false if the uniqueness requirement is dropped. %
Under regularity conditions establishing sufficient continuity,
Bayes procedures with respect to priors with full support are admissible.
This condition is not necessary.
The most widely used method for showing admissibility---known as Blyth's method---is 
a generalization of this result based on a theorem of Blyth. (See \cite{lehmann-casella}, Theorems~7.13 and~8.7, also labeled Problem~7.13, for two of its variants.) 
Blyth's theorem formulates a sufficient criterion for being admissible in terms of rates of convergence of certain risks associated to 
a sequence of Bayes procedures approximating a procedure with continuous risk function.
It is not known whether Blyth's method is both a sound and \emph{complete} approach to establishing admissibility.

While there is a relative dearth of necessary and sufficient conditions for admissibility in terms of Bayes optimality,
Stein \cite{stein1955} formulated matching conditions for decisions problems satisfying strong regularity conditions related to compactness as well as a family of game-theoretic conditions.
Stein formulated a second version of his condition in \cite{stein1965}, which requires the same regularity assumptions: Namely, (1) a family of derived statistical decision problems must be determined in the sense that the minimax and maximin values must agree, and (2) the set of decision procedures must satisfy a compactness condition dubbed ``weak compactness in the sense of Wald'' by Stein. 
Once more, examples show one cannot drop either of the regularity assumptions made by Stein.
Stein's condition can be interpreted as an indication that the criterion from Blyth's method is also necessary---that is, if a procedure is admissible, there is a proof by Blyth's method. 
Indeed, this follows from Stein's work for sufficiently regular decision problems (see Section~\ref{ss.stein} below).

\medskip

Recently, methods from nonstandard analysis have been brought to bear on problems in statistical decision theory.
In \cite{duanmuroy}, two of the present authors have identified an exact characterization of \emph{extended admissibility} in terms of Bayes optimality, although one must (1)  enrich the set of available priors to allow priors that assign ``infinitesimal'' mass to sets and (2) ignore infinitesimal differences in Bayes risk when judging Bayes optimality. 
In the present article we continue this line of work, but turn our attention from extended admissibility to admissibility.

With Theorem~\ref{t.sconv} we give a necessary and sufficient condition for admissibility
in term of Bayes optimality (without any regularity assumptions).
In any decision problem, admissibility is Bayes optimality when the latter is appropriately interpreted 
in the setting of nonstandard analysis.
In the simplest possible terms, a procedure is admissible \emph{iff} it is Bayes w.r.t.\ a prior which gives positive---and perhaps infinitesimal---weight to every point in the parameter space.

The main ingredients to the proof of this characterization are:  
First, methods from nonstandard analysis such as the use of model-theoretic saturation; 
Second, a classical characterization of admissibility in problems which are, in a certain sense, finite; and, Third, that admissibility among finitely many procedures is witnessed by a finite set of parameter points.

It would be wrong to regard our characterization as ``Bayes optimality with respect to a weakened notion of prior'' (similar to the use of improper priors).
The correct view is that as we admit the so-called hyperreals into our number system
and extend functions and measures similarly.
More precisely, to each mathematical object corresponds its image under the so-called star map, or monomorphism.
This allows us to characterize the admissible procedures as $\starmap{}$Bayes among \emph{the external set of standard procedures} with respect to a so-called internal prior. See Section~\ref{s.ns} for a quick review of these notions, and of nonstandard analysis.

\medskip

The exact characterization from Theorem~\ref{t.sconv} leads to a variant of Stein's condition (see Definition~\ref{d.stein.blyth}, Items~\ref{d.i.stein.B} and~\ref{d.i.noB} and Theorem~\ref{t.adm}, Items~\ref{i.stein} and~\ref{i.stein.B}) which is necessary and sufficient for admissibility in full generality, that is, with no assumptions whatsoever on the decision problem.
Moreover we find that a variant of Blyth's method for showing admissibility is also equivalent to admissibility (see Definition~\ref{d.stein.blyth}, Items~\ref{d.i.blyth.B} and~\ref{d.i.noB} and Theorem~\ref{t.adm}, Items~\ref{i.blyth} and~\ref{i.blyth.B}).
That is, when couched in a nonstandard setting, Blyth's method must in principle suffice to prove admissibility for any statistical decision procedure.

These results do not only clarify the role of the regularity assumptions in the many classical theorems relating 
admissibility and Bayes optimality. %
Emboldened by the knowledge that the nonstandard Blyth method is a sound and complete method for establishing admissibility, 
we revisit its application to the Graybill--Deal estimator. 
We show that this estimator is admissible
among a certain class of %
unbiased estimators (Corollary~\ref{c.GD}).

Finally, our methods also lead to a new characterization of admissibility under regularity assumptions, in an entirely ``standard'', classical setting---that is, not mentioning any objects from non-standard analysis.
Namely, we also show that in problems with bounded risk functions, admissibility is equivalent to being Bayes under a finitely additive prior (Theorem~\ref{t.bounded.risk}).

\subsection*{Acknowledgments}
We acknowledge the support of the Government of Canada’s New Frontiers in Research Fund (NFRF), 
NFRFE-2018-02164.
The third author would moreover like to thank the Austrian Science Fund (FWF) for the generous support
through START Grant Y1012-N35. 

\section{Some technicalities}

\subsection{Decision problems and regularity}

For clarity, we first discuss and state explicitly the background assumptions we make to guarantee that in decisions problems as we consider them, integrability conditions and similar technical issues pose no further distraction.

To clarify what me mean by a \emph{decision problem}, consider  
a sextuple  
\begin{equation}\label{e.decisionproblem}
\aproblem= \decisionproblem
\end{equation}
where 
\begin{enumerate}[label=(\roman*)]
\item\label{sextuple.first} $\parametersp$, the parameter space, 
is a measurable space which we think of as the possible states of nature;
 \item $\actionsp$, the action space, is a measurable space which we think of the possible courses of action available to the statistician;
\item  $\samplesp$, the sample space is a measurable space which we think as consisting of the possible observations made by the statistician  before deciding on a course of action;
\item upon making a (non-randomized) choice, she suffers a loss given by the 
measurable
function 
$
\lossf \colon \parametersp \times \actionsp \to (-\infty,+\infty);
$
\item $\model = (P_\theta)_{\theta\in\parametersp}$, the model determining the probabilities of observations, is a probability kernel from $\parametersp$ to $\samplesp$; i.e., 
we can view $\model$ as measurable function
\begin{gather*}
\model \colon \parametersp \to \prob(\samplesp),\\
\theta \mapsto \model_\theta;
\end{gather*} 
\item the set of decision procedures
$\decisions$ is a set of 
probability kernels from $\samplesp$ to $\actionsp$, or in other words, a set of 
measurable functions from $\samplesp$ into the set of probability measures on $\actionsp$;
\item\label{sextuple.last} $\priors$, which we think of as the set of priors which the statistician is willing to consider, is a set of probability measures on $\parametersp$.\footnote{Of course, one could be more general and allow $\sigma$-finite or arbitrary measures, and/or demand only finite additivity here.}
\end{enumerate} 

Note that under the above assumptions, the expected loss or risk of $\delta \in \decisions$ given $\theta \in \parametersp$, i.e., the following integral\footnote{One should really write $r^\aproblem$ since $r$ depends on (the first five compononents of) the problem $\aproblem=\decisionproblem$ under consideration, but we will follow general practice and omit reference to $\aproblem$.}
\begin{equation}\label{e.risk}
\riskf(\theta,\delta) := \int_{\samplesp}\int_{\actionsp}\lossf\big(\theta,a\big)\; \delta(x,\mathrm{d}a) P_\theta(\mathrm{d}x) = \expected{x\sim P_\theta, a \sim \delta(x) } \lossf(\theta,a)
\end{equation}
need not be defined or finite; 
and 
the Bayes risk of $\delta \in \decisions$ under a prior $\pi \in \priors$, i.e., the following integral 
\begin{equation}\label{e.bayesr}
\riskf(\pi,\delta) = \int_{\parametersp}\riskf(\theta,\delta) \pi(\mathrm{d}\theta) =  \expected{\theta \sim \pi} \riskf(\theta,\delta)
\end{equation} 
need not be well-defined, either; in fact, not even the function under the integral need be everywhere defined.
The conventional remedy is to impose regularity conditions of one kind or another, to ensure the integrals above exist.\footnote{More assumptions, such as restricting $\decisions$ and/or $\priors$, are commonly made  so as to ensure
that posteriors are well-defined. We have no need for this.}
This introduces a trade off: One may aim for a particular properties---such as being closed with respect to some convenient topology---for the set of priors or the set of decision procedures; or one may aim for imposing as few conditions as possible on $\parametersp$, $\actionsp$, $\lossf$ and the model.

\medskip

We shall henceforth assume the following, without further mention:
\begin{enumerate}[label=(\Roman*)]
\item\label{a.first} The loss function is bounded from below.
\item\label{a.risk} For any $\theta \in \parametersp$ and any $\delta \in \decisions$ the risk, i.e., the integral in \eqref{e.risk}
is finite. %
\item\label{a.bayesrisk} For each $\pi \in \priors$ the Bayes risk, i.e., the integral in  \eqref{e.bayesr}
is finite.
\item\label{a.last} The set of priors $\priors$ contains all priors with finite support.
\end{enumerate}

The reader may see these assumptions as part of our definition of ``decision problem''.
In fact, for better readability, let us state the following definition.
\begin{definition}
In this article, by a \emph{decision problem} we mean a sextuple as in \eqref{e.decisionproblem} such that \ref{sextuple.first}--\ref{sextuple.last} and \ref{a.first}--\ref{a.last} hold.
\end{definition}

Note that in the context of our other assumptions, \ref{a.risk} above implies that 
\begin{gather*}
\riskf(\cdot, \delta)\colon \parametersp \to \reals,\\
\theta \mapsto \riskf(\theta,\delta) %
\end{gather*}
is measurable for every $\delta \in \decisions$, but in general \ref{a.bayesrisk} is not automatic.

\subsection{Nonstandard Analysis}\label{s.ns}
In this section, we quickly review notation and terminology from nonstandard analysis.
Readers with no prior knowledge of nonstandard analysis are referred to \cite{loeb-working}, \cite{nato-book}, or \cite[Appendix A]{duanmuroy} for an introduction.

Throughout, we shall fix a \emph{nonstandard extension}, given by a map, the \emph{star map}, 
\[
x \mapsto \starmap x
\]
taking any object to its nonstandard counterpart or \emph{extension}.
While there is in principle nothing to stop us from taking the entire universe of sets as the domain of this function, we follow the usual convention and take its domain to be a \emph{superstructure} of the form $V_\omega(X) = \bigcup_{n\in\omega} V_n(X)$, where $X$ is some set whose elements we regard as being fixed by the star map, and where $V_n(X)$ is defined inductively by $V_0(X) = X$ and 
$V_{n+1}(X) = \powerset(V_n(X))\cup V_n(X)$.
We assume $\reals \subseteq X$.
The star map is into $V^*_\omega(\starmap X)$, called the \emph{nonstandard universe} and it is an elementary embedding:
A first order formula $\phi(x_0, \hdots, x_n)$ holds in the structure $\langle V_\omega(X), \in\rangle$ if and only if 
$\starmap \phi := \phi(\starmap x_0, \hdots, \starmap x_n)$ holds in $\langle V^*_\omega(\starmap X), \in \rangle$.
A set $x$ is called internal if $x \in V^*(X)$, and external if $x \subseteq \starmap(V_n(X))$ for some $n\in\omega$.

For often used relations (such as $\leq$) and functions (such as $+$) we often drop the star when we refer to the image under the star map, that is, we write
$\leq$ instead of $\starmap\leq$ and $+$ instead of $\starmap +$. 
In practice, this rarely causes confusion.

We shall further assume that $V^*_\omega(\starmap X)$ is polysaturated (or if you prefer, $\kappa$-saturated for some large enough $\kappa$). Recall:
\begin{definition}
The nonstandard extension is called $\kappa$-saturated if given any (internal or external) set $\Theta$ of size $<\kappa$ consisting of internal sentences having $x$ as their only free variable, $\Theta$ is satisfiable in $V^*_\omega(\starmap X)$ if it is finitely satisfiable.
Equivalently, every Filter of size $<\kappa$ consisting of internal sets has non-empty intersection.

The nonstandard extension is called \emph{polysaturated} if it is $\kappa$-saturated where $\kappa$ is the cardinality of $V_\omega(X)$.
\end{definition}
If the reader wishes to use the weakest possible assumption, in this article they can work under $\kappa$-saturatedness, where $\kappa$ is the supremum of the cardinalities of all sets of decision procedures and of all parameter spaces they are interested in considering.

\medskip

Given a topological space $X$, the \emph{near standard points of $X$}, $\NS{X}$ are the points $\tilde x \in \starmap X$ for which there exists $x \in X$ such that $\tilde x \in \starmap U$ for every open neighbourhood $U$ of $x$.
If $\tilde x \in \NS{X}$ and $X$ is Hausdorff, we write $\ST\tilde x$ for the unique $x$ as above.
It is well-known that a topological space is compact if and only if $\starmap X = \NS{X}$.
In particular, the map $\tilde x \mapsto \ST{\tilde x}$ is defined on $\starmap [0,1]$, or more generally, on $\NS{\starmap \reals}$.
Elements of $\starmap \reals$ are called hyperreals, and internal sets which can be brought into internal bijection with the set of predecessors of an element of $\starmap \nat$ are called hyperfinite.
For hyperreals, we also say \emph{finite} to mean near standard. 
A hyperreal $\tilde r$ is \emph{infinitesimal} if it is finite and $\ST \tilde r = 0$.
Given $\tilde r, \tilde r' \in \starmap \reals$, we write $\tilde r \approx \tilde r'$ to mean that $\tilde r-  r'$ is infinitesimal.
We write $\tilde r \lessapprox \tilde r'$ to mean that $\tilde r \leq \tilde r' + \tilde\varepsilon$ for some infinitesimal $\tilde\varepsilon$
(of course, $\gtrapprox$ is defined similarly).

\section{Hyperpriors and admissibility}

For our characterization of admissibility, let us introduce the following terminology.

\begin{definition}
Given a decision problem $\decisionproblem$, $\Pi \in \starmap\priors$, $\Delta_0 \in \starmap\decisions$ and a possibly external set $\decisions_0 \subseteq \starmap\decisions$, we say 
\emph{$\Delta_0$ is $\starmap{}$Bayes among $\decisions_0$ under $\Pi$} or \emph{$\Pi$-$\starmap{}$Bayes among $\decisions_0$} to mean that
for all $\Delta \in \decisions_0$,
$\starmap\riskf(\Delta_0,\Pi) \leq \starmap\riskf(\Delta,\Pi)$. 
We say 
\emph{$\Delta_0$ is $\starmap{}$Bayes among $\decisions_0$ among $\decisions_0$} to mean that for some $\Pi \in \starmap\priors$, $\Delta_0$ is $\starmap{}$Bayes among $\decisions_0$ under $\Pi$.

We write 
\[
\sigmamap{\decisions} := \{\starmap\delta \setdef \delta \in\decisions\},
\]
called the \emph{standard part copy of $\decisions$}.
\end{definition}

We can now state our theorem giving a characterization of admissibility in Bayesian terms.

\begin{theorem}\label{t.sconv}
Suppose $\decisionproblem$ is a decision problem %
where $\decisions$ is a convex set of decision procedures, and $\delta_0 \in \decisions$ is admissible  among $\decisions$.
Then there exists a non-standard prior $\Pi \in \starmap{\priors}$ with hyperfinite support such that $\Pi(\{\starmap{\theta}\})>0$ for every $\theta\in\parametersp$ and such that
$\starmap{\delta_0}$ is a $\Pi$-$\starmap{}$Bayes procedure among $\sigmamap{\decisions}$.
\end{theorem}
A little less pedantically, 
$\delta_0$ is %
Bayes in $\decisions$ with respect to a non-standard prior which gives (possibly infinitesimal) positive weight to all standard points. 

\medskip

As we shall see, the above theorem follows easily from a preliminary result involving only standard notions, which we prove in several steps in the next section. 
In the section after that, we will prove the above theorem.

\subsection{Finite sets of parameters witnessing admissibility}

In this and the next section, we show that any admissible decision procedure can be separated in risk from finitely many other decision procedures with respect to a carefully chosen prior.
In fact it be can done using a prior which is finitely supported.
For a ``standard proof'' of this result (that is, a proof not using infinitesimals) we shall first prove the following two results as stepping stones.
We find them interesting in their own right and so state them separately. 
After that, we shall give a second proof using nonstandard analysis. 

\begin{lemma}\label{l.steppingstone}
Let $\decisionproblem$ be a decision problem and $\delta_0 \in \decisions$.
Suppose one of the following:
\begin{enumerate}[label=(\roman*),ref=(\roman*)]
\item\label{i.compact} The set $\decisions' \subseteq \decisions$ can be given the structure of a compact topological space such that $\delta \mapsto \riskf(\theta,\delta)$ is continuous on $\decisions'$ for every $\theta \in\parametersp$. %

\item\label{i.randomized} $\decisions'\subseteq\decisions$ arises as the set of randomizations of a finite subset $\decisions_1$ of $\decisions$. %
\end{enumerate}
Suppose further that $\delta_0$ is not in $\decisions'$ up to equivalence in risk and admissible among $\decisions'$.
Then there is a finite set $\parametersp_0 \subseteq \parametersp$
such that for every $\delta \in \decisions'$ there is $\theta\in\parametersp_0$ with
$\riskf(\theta,\delta_0) < \riskf(\theta,\delta)$.
In other words, $\delta_0$ is admissible among $\decisions'$ witnessed on $\parametersp_0$.
\end{lemma}

\begin{proof}[Proof of Lemma~\ref{l.steppingstone}]
Fix $\decisionproblem$, $\decisions'$ and $\delta_0$ as in the lemma.
As we shall see, the case described in each item leads to a proof of the next one.

First assume the hypothesis in Item \ref{i.compact} holds.
By admissibility and since $\delta_0$ is not equivalent to any procedure in $\decisions'$, for every $\delta \in \decisions'$ there is $\theta \in \parametersp$ 
such that $\riskf(\theta,\delta_0) < \riskf(\theta,\delta)$.
Since the risk function is continuous in its second argument, we can find for each $\delta \in \decisions'$ a neighbourhood $U$ of $\delta$ such that the same $\theta$ works for each $\delta' \in U$, i.e., 
such that for each $\delta'\in U$ it holds that $\riskf(\theta,\delta_0) < \riskf(\theta,\delta')$.
By compactness of $\decisions'$ we can cover $\decisions'$ by finitely many neighbourhoods $U_0, \hdots, U_n$, and find
$\theta_0, \hdots, \theta_n \in \parametersp$ so that for every $\delta \in \decisions'$ there is $i\leq n$ such that $\delta \in U_i$, and 
$\riskf(\theta_i,\delta_0) < \riskf(\theta_i,\delta')$ for every $\delta' \in U_i$.
In particular, for every $\delta \in \decisions'$, there is a parameter $\theta_i$ in
\[
\parametersp_0 := \{\theta_0, \hdots, \theta_n\}
\]
such that $\riskf(\theta_i,\delta_0) < \riskf(\theta_i,\delta)$.

\medskip

Next, let us assume the hypothesis from Item \ref{i.randomized}. 
Let us enumerate $\decisions_1$ as follows,
\[
\decisions_1 = \{\delta_1, \hdots, \delta_n\},
\]
recalling that by assumption, $\decisions'$ is the set of randomizations of procedures from $\decisions_1$.

We can view $\decisions'$ as a compact topological space by identifying each randomized procedure with a point in the $n$-simplex $\Delta^{n-1}$. 
In a decision problem with randomization the risk function is linear in the second argument for convex sums, and so
 for each $\theta \in \parametersp$, the function 
\begin{gather*}
\delta \mapsto \riskf(\theta, \delta),\\
\Delta^{n-1} \to \reals
\end{gather*} 
is continuous. So we can apply the previous result, Item~\ref{i.compact}, of the present lemma to obtain $\parametersp_0$ as claimed.
\renewcommand{\qedsymbol}{{\tiny Lemma \ref{l.steppingstone}.} $\Box$}
\end{proof}

\begin{remark}
The previous lemma uses only very little of the structure of a decision problem, and in particular, could also be formulated in a game theoretic setting.
Item~\ref{i.compact} will go through for any sets $\parametersp$ and $\decisions$ and any function 
$\riskf \colon \parametersp \times \decisions$ satisfying the continuity assumption stated in the lemma.
Item~\ref{i.randomized}  will go through for any set $\parametersp$ and function $\riskf \colon \parametersp \times \decisions$, provided $\decisions$ is convex and $\riskf$ is linear in its second argument.
\end{remark}

As it happens, the part of Lemma~\ref{l.steppingstone} which is relevant in what follows, namely Item \ref{i.randomized}, can also be shown directly using a non-standard argument.

\begin{proof}[Nonstandard proof of Item \ref{i.randomized} of Lemma~\ref{l.steppingstone}]
We shall show that there exists a finite set $\parametersp_0\subseteq \Theta$ as above by contradiction; 
so let us suppose not. 
Then for any finite subset $\parametersp'\subseteq \Theta$, there exists a decision procedure $\delta_{\parametersp'}\in \decisions'$ such that $\delta_{\parametersp'}$ dominates $\delta_0$ on $\parametersp'$.  For each $\parametersp'\in \Theta^{[<\omega]}$, let $\phi(\parametersp')$ be the formula
\[
(\exists \Delta\in \starmap{\decisions'})\;\Delta \text{ dominates } \starmap{\delta_0} \text{ on } \parametersp'.
\] 
Then the collection $\{\phi(\parametersp'): \parametersp'\subseteq\parametersp$, $\parametersp'$ finite $\}$ is finitely satisfiable. By saturation, 
we can find $\Delta\in \starmap{\decisions'}$ such that $\Delta$ dominates $ \starmap{\delta_0}$ on $\parametersp$.

Clearly, $\Delta$ is a $\starmap{}$convex combination of elements in 
$\starmap{\decisions_1}$.  
Let $\delta_1, \hdots, \delta_l$ enumerate $\decisions_1$. 
Find $\lambda_1, \hdots, \lambda_n \in\starmap{[0,1]}$ such that  $\Delta=\sum_{i=1}^{n}\lambda_{i}\starmap{\delta_i}$. 
Define a standard decision procedure $\delta$ as follows:
\[
\delta:=\sum_{i=1}^{n}\ST(\lambda_{i})\delta_i.
\] 
Then $\delta$ is an element of $\decisions'$ and $\riskf(\theta, \delta)\approx\starmap{\riskf(\theta, \Delta)}$.
Since by  construction, $\Delta$ dominates $\delta_0$ on $\parametersp$, i.e., $\starmap{\riskf(\theta, \Delta)}\leq \starmap{\riskf(\theta, \starmap{\delta_{0}})}$  for every $\theta\in \parametersp$, and
since in a decision problem with randomization, the risk function is linear in the second argument for convex sums, we infer that also $\riskf(\theta,\delta)\leq \riskf(\theta, \delta_0)$ for every $\theta\in \parametersp$. 
This contradicts our assumption that $\delta_0$ is admissible among but not equivalent in risk to a procedure in $\decisions'$, finishing the proof.
\end{proof}

\subsection{Admissibility is Bayes optimality with respect to a hyperprior}

Using the arguments of the previous section, we can reduce the problem of proving Theorem~\ref{t.sconv} 
to a problem involving only a much smaller, very well studied class of decision problems.
In this class, the theorem is known to be true:

\begin{theorem}[\cite{blackwell-girshick}, 5.2.5, p.~130]\label{t.hyperprior.admissibility}
Suppose $\decisionproblem$ is a decision problem such that $\parametersp$ is finite and 
$\decisions$ is the convex hull of finitely many decision procedures, and further suppose $\delta_0 \in \decisions$ is
admissible among $\decisions$.
Then there exists a prior $\pi \in \priors$ such that $\delta_0$ is $\pi$-Bayes among $\decisions$ and
$\pi$ is everywhere positive on $\parametersp$. 
\end{theorem}

We now have all the necessary tools at our disposal to prove our main theorem.
\begin{proof}[Proof of Theorem~\ref{t.sconv}]
Let $\decisionproblem$ and $\delta_0$ be given as in Theorem~\ref{t.sconv}.
Using saturation, find hyperfinite sets $\tilde\parametersp$ and $\tilde \decisions$ such that
\begin{equation}\label{e.tilde}
\begin{gathered}
\parametersp\subseteq \tilde\parametersp \subseteq \starmap{\parametersp},\\
\decisions\subseteq \tilde\decisions \subseteq \starmap{\decisions}.
\end{gathered}
\end{equation}
Note that $\starmap{\delta_0}$ is admissible among the $\starmap{}$convex hull of $\tilde\decisions$.
By transfer and Lemma~\ref{l.steppingstone} Item \ref{i.randomized}, we may assume that $\tilde\parametersp$ is large enough to witness that 
$\starmap{\delta_0}$ is admissible among the $\starmap{}$convex hull of $\tilde\decisions$ (otherwise, replace $\tilde\parametersp$ with its union with a hyperfinite set $\tilde\parametersp_0 \subseteq \starmap{\parametersp}$ obtained by using the transfer of Lemma~\ref{l.steppingstone} Item \ref{i.randomized}).
By transfer, the previous theorem holds in the nonstandard extension for problems with hyperfinite parameter space and convex sets of procedures with hyperfinitely many extreme points.
Thus, it applies to the problem obtained from $\starmap{\decisionproblem}$ by restricting the space of parameters to  $\tilde\parametersp$ and the set of allowed procedures to the $\starmap{}$convex hull of $\tilde \decisions$ 
(and restricting $\lossf$ and $\priors$ in the obvious manner).
We obtain a hyperprior $\Pi \in \starmap{\priors}$ such that $\tilde \parametersp = \supp(\Pi)$ and
$\starmap{\delta_0}$ is $\starmap{}$Bayes among $\tilde \decisions$.
By \eqref{e.tilde} this proves the theorem.
\end{proof}

\section{Stein's condition}\label{s.stein}
\subsection{Stein's necessary and sufficient condition for admissibility}\label{ss.stein}

Let's introduce the following short-hand for risk functions shifted so that they are relative to the base-line given by a procedure $\delta$.
\begin{definition}
For $\delta, \delta' \in \decisions$, and $\pi \in\priors$ write
\[
\riskf_\delta(\pi,\delta') := \riskf(\pi,\delta') - \riskf(\pi,\delta).
\]
\end{definition}
Of couse likewise, for $\theta\in\parametersp$ use the corresponding short-hand,
\[
\riskf_\delta(\theta,\delta') := \riskf(\theta,\delta') - \riskf(\theta,\delta),
\]
which one can view as a special case of the previous definition via the embedding of $\parametersp$ in $\priors$ by associating to each parameter point its Dirac measure.

\medskip

In his 1955 article \cite{stein1955} Stein formulated a condition (which we shall call \emph{Stein's condition}) 
which characterizes admissibility provided the decision problem satisfies certain technical assumptions (which we shall call
\emph{Stein's assumptions}).

\begin{definition}
Let $\decisionproblem$ be a decision problem. 
We say $\delta_0 \in \decisions$ satisfies \emph{Stein's condition} if
\[
(\forall \theta \in \parametersp) (\forall \varepsilon \in \reals_{>0})
(\exists \pi \in \priors)\;  \left[
\pi(\{\theta\})>0 \land
\riskf(\pi,\delta_0) - \inf_{\delta \in \decisions} \riskf(\pi,\delta) \leq \pi(\{\theta\}) \cdot  \varepsilon \right]
\]
\end{definition}
This condition implies admissibility (without any assumption on the decision problem at hand).
To state Stein's technical assumptions sufficient for its sufficiency, we first recall Wald's notion of weak compactness:
\begin{definition}
Given a decision problem $\decisionproblem$, we say $\decisions$ is weakly compact
if every sequence $(\delta_n)_{n\in \nat}$ from $\decisions$ has a subsequence $(\delta_{n(k)})_{k\in \nat}$ such that there exists 
$\delta_\infty$ with
\[
(\forall \theta\in\parametersp) \; \liminf_k \riskf(\theta,\delta_{n(k)}) \geq  \riskf(\theta,\delta_\infty)
\]
\end{definition}

\begin{definition}
We say a decision problem $\decisionproblem$ satisfies \emph{Stein's assumptions} if
the game $(\priors,\decisions, \riskf)$, associated to this problem satisfies the following:
\begin{enumerate}
\item $\decisions$ is weakly compact in the sense of Wald,
\item For every $\theta_0 \in \parametersp$ and every $\gamma \in \reals_{>0}$ %
we have %
\begin{equation}
\inf_{\delta\in \decisions} \sup_{\theta\in \parametersp} 
\big[
r_{\delta_0}(\theta_0, \delta) + \gamma r_{\delta_0}(\theta, \delta) 
\big]
= \sup_{\theta\in \parametersp} \inf_{\delta\in \decisions}
\big[
r_{\delta_0}(\theta_0, \delta) + \gamma r_{\delta_0}(\theta, \delta)
\big]
\end{equation}
that is, the game ($\priors, \decisions, r^{\theta_0,\gamma})$ with
$r^{\theta_0,\gamma} = r_{\delta_0}(\theta_0, \delta) + \gamma r_{\delta_0}(\theta, \delta)$
has a value.
\end{enumerate}
\end{definition}

We can now state Stein's classical result as follows:

\begin{theorem}[\cite{stein1955}]
Provided Stein's assumptions hold for $\decisionproblem$, $\delta_0$ satisfies Stein's criterion if (and therefore, if and only if) $\delta_0$ is admissible among $\decisions$.
\end{theorem}

\subsection{Stein's condition with hyperpriors}
We now show how to derive a nonstandard version of Stein's characterization for arbitrary decision problems.
At the same time, we show that a nonstandard version Blyth's Condition also is necessary and sufficient for admissibility; 
here, Blyth's Conditions refers to a sufficient condition for admissibility used in what is often called Blyth's method (see \cite[Theorem~7.13, Theorem~8.7/Problem~7.12]{lehmann-casella}) and which is the basis of many proofs of admissibility in the literature.

\medskip

To state these theorems, it will be convenient to use the following terminology.

\begin{definition}\label{d.stein.blyth} Let $\aproblem = \decisionproblem$ be a decision problem. 
\begin{enumerate}
\item\label{d.i.B} We say that $\mathcal B$ is a \emph{determining family} (for the problem $\aproblem$) to mean that $\mathcal B$ is a family of non-empty subsets of $\parametersp$ such that for any $\delta_0, \delta_1 \in \decisions$, if there is $\theta_0 \in \parametersp$ such that
$\riskf(\theta_0,\delta_1) < \riskf(\theta_0,\delta_0)$, then there is $B\in \mathcal B$ and $\varepsilon \in \reals_{>0}$ such that
\[
(\forall \theta \in B)\; \riskf(\theta,\delta_1) < \riskf(\theta,\delta_0)-\varepsilon.
\]

\item\label{d.i.stein.B} Let $\mathcal B \subseteq \powerset(\parametersp)$. We say that \emph{$\delta_0 \in \decisions$ satisfies the nonstandard Stein's condition with respect to $\mathcal B$} if for all $B \in \mathcal B$ and all $\varepsilon \in \reals_{>0}$ there exists a hyperprior $\Pi \in \starmap{\priors}$ 
such that 
\[
\starmap{\riskf}(\Pi,\starmap{\delta_0}) - \inf_{\Delta \in \sigmamap{\decisions}} \starmap{\riskf}(\Pi,\Delta) \leq \Pi(\starmap B) \cdot \varepsilon
\]

\item\label{d.i.blyth.B} We say that \emph{$\delta_0 \in \decisions$ satisfies the nonstandard Blyth's condition with respect to $\mathcal B$} if there exists a hyperprior $\Pi \in \starmap{\priors}$ and  
$\tilde \rho \in \starmap\reals_{>0}$ such that for all $B \in \mathcal B$ there exists $C \in \reals_{>0}$ such that $\tilde\rho \leq  C\cdot\Pi(\starmap B)$ and
\[
\frac{\starmap{\riskf}(\Pi,\starmap{\delta_0}) - \inf_{\Delta \in \sigmamap{\decisions}} \starmap{\riskf}(\Pi,\Delta) }{\tilde \rho} \lessapprox 0
\] 

\item\label{d.i.noB} We say simply that \emph{$\delta_0 \in \decisions$ satisfies the nonstandard Stein's condition} (with no mention of a family $\mathcal B$) to mean that 
$\delta_0 \in \decisions$ satisfies the nonstandard Stein's condition with respect to the family $\mathcal B \subseteq \powerset(\parametersp)$ consisting of the singletons.
Likewise, to say that \emph{$\delta_0 \in \decisions$ satisfies the nonstandard Blyth's condition} means that
$\delta_0 \in \decisions$ satisfies the nonstandard Blyth's condition with respect to the family $\mathcal B \subseteq \powerset(\parametersp)$ consisting of the singletons.
\end{enumerate}
\end{definition}

\begin{theorem}\label{t.adm}
Let $\aproblem=\decisionproblem$  be a decision problem and $\mathcal B$ any determining family for the problem $\aproblem$.
The following are equivalent:
\begin{enumerate}
\item\label{i.adm} $\delta_0$ is admissible among $\decisions$.
\item\label{i.stein} $\delta_0$ satisfies the nonstandard Stein's condition. %
\item\label{i.stein.B} $\delta_0$ satisfies the nonstandard Stein's condition with respect to $\mathcal B$.
\item\label{i.blyth} $\delta_0$ satisfies the nonstandard Blyth's condition. %
\item\label{i.blyth.B}  $\delta_0$ satisfies the nonstandard Blyth's condition  with respect to $\mathcal B$. 
\end{enumerate}
\end{theorem}

\begin{proof}
(\ref{i.adm}) $\Rightarrow$ (\ref{i.blyth}): Since $\delta_0$ is admissible and by Theorem~\ref{t.hyperprior.admissibility} we can find a hyperprior $\Pi$ with hyperfinite support 
such that 
$\starmap{\delta_0}$ is $0$-$\starmap{}$Bayes among $\sigmamap{\decisions}$ and $\parametersp \subseteq \supp(\Pi)$. 
Let $\tilde \rho = \min\{\Pi(\theta) \setdef \theta \in \supp(\Pi)\}$.
Since also
\[
\frac{\starmap{\riskf}(\Pi,\starmap{\delta_0}) - \inf_{\Delta \in \sigmamap{\decisions}} \starmap{\riskf}(\Pi,\Delta) }{\tilde \rho}=0,
\] 
clearly $\delta_0$ satisfies the nonstandard Blyth condition (that is, it satisfies the nonstandard Blyth condition with respect to the family of singletons from $\parametersp$).

(\ref{i.blyth}) $\Rightarrow$ (\ref{i.blyth.B}): 
Suppose that $\Pi$ and $\tilde \rho$ witness the nonstandard Blyth condition with respect to a family $\mathcal B'$ of subsets of $\parametersp$ and that for any $B\in \mathcal B$ there is $B' \in \mathcal B'$ such that $B' \subseteq B$.
Then $\Pi$ and $\tilde \rho$ witness the nonstandard Blyth condition with respect to the family $\mathcal B$.
In particular, if $\Pi$ and $\tilde \rho$ witness the nonstandard Blyth condition (with respect to the set consisting of singletons from $\parametersp$) the same is true with respect to $\mathcal B$.

(\ref{i.blyth}) $\Rightarrow$ (\ref{i.stein}) and (\ref{i.blyth.B}) $\Rightarrow$ (\ref{i.stein.B}): 
It is enough to show the second implcation; the first is the special case of the second where we take $\mathcal B$ equal to the set of singletons from $\parametersp$. 
So let $\Pi$ and $\tilde \rho$ 
witness the nonstandard Blyth's condition for $\mathcal B$.
Then given any $B \in \mathcal B$ we can find a positive infinitesimal $\tilde\varepsilon$ and $C\in\reals_{>0}$ such that
\[
\starmap{\riskf}(\Pi,\starmap{\delta_0}) - \inf_{\Delta \in \sigmamap{\decisions}} \starmap{\riskf}(\Pi,\Delta) \leq\tilde \rho \cdot \tilde\varepsilon \leq \Pi(\starmap B) \cdot C \cdot  \tilde\varepsilon 
\] 
Since for any $\varepsilon \in\reals^{>0}$ we have $C\tilde\varepsilon <\varepsilon$, Stein's condition holds for $\mathcal B$.

(\ref{i.stein}) $\Rightarrow$ (\ref{i.adm})  and (\ref{i.stein.B}) $\Rightarrow$ (\ref{i.adm}): As above, it is enough to prove the second implication.
Although $\Pi$ is not a standard prior, the usual proof goes through almost verbatim:
Suppose towards a contradiction that $\delta \in \decisions$ and $\delta \prec \delta_0$. 
Since $\mathcal B$ is a determining family, we can find $B \in \mathcal B$ and $\varepsilon \in \reals^{>0}$ such that $\riskf(\theta,\delta)<\riskf(\theta,\delta_0) - \varepsilon$ for every $\theta \in B$.
Find a hyperprior $\Pi$ witnessing (\ref{i.stein.B}) for
$\varepsilon$.
By choice of $B$ and $\varepsilon$, we have
 \begin{multline*} 
\starmap{\riskf}(\Pi,\starmap{\delta_0}) - \inf_{\Delta \in \sigmamap{\decisions}} \starmap{\riskf}(\Pi,\Delta)
\geq
\starmap{\riskf}(\Pi,\starmap{\delta_0}) -  \starmap{\riskf}(\Pi,\starmap{\delta})
=
\starmap{\riskf}_{\starmap{\delta}}(\Pi,\starmap{\delta_0})
= \\
\Pi(\{\starmap\theta_0\}) \cdot \riskf_{\delta}(\theta_0,\delta_0) 
+ \starmap{\riskf}_{\starmap{\delta}}(\Pi \cdot 1_{\starmap{\parametersp}\setminus\{\theta_0\}},\starmap{\delta_0})
\geq \\
\Pi(\{\starmap\theta_0\}) \cdot \riskf_{\delta}(\theta_0,\delta_0)  = 
 \Pi(\{\starmap\theta_0\}) \cdot \varepsilon
\end{multline*}
where the last inequality holds since by transfer, $\starmap\riskf_{\starmap\delta}(\tilde\theta,\starmap\delta_0) \geq 0$ for all $\tilde\theta \in \starmap\parametersp$.
This contradicts that $\Pi$ should witness (\ref{i.stein.B}).
\end{proof}

\section{Examples and Applications}

\subsection{Relative admissibility of the Graybill-Deal estimator}\label{s.GD}

Let $n>1$ and consider the problem of estimating the common mean $\mu$ of random variables $X^{(1)}_1, \hdots, X^{(n)}_1$ and 
$X^{(1)}_2, \hdots, X^{(n)}_2$, where for each $i\in\{1,2\}$ the random variables $X^{(1)}_i, \hdots, X^{(n)}_i$ are i.i.d.\ 
according to $\normaldist(\mu,\sigma_i)$ with unknown variance $\sigma^2_i$.
Consider the estimator 
\[
\gdest = \left(\sum_{i=1}^{m}\frac{n}{S^2_{i}}\right)^{-1} \cdot
\sum_{i=1}^{2} \frac{n}{S^2_{i}}\bar X^i
\]
where of course $\bar X_i$ is the mean of $X^{(1)}_i, \hdots, X^{(n)}_i$, and where
\[
S^{2}_{i} := \frac{1}{n-1} \sum_{j=1}^{n}(X^{(j)}_i - \bar X_i)^2
\]
This is 
a special case of the Graybill-Deal estimator, proposed in \cite{graybill-deal} . 

In \cite{sinha1982} it is asked whether $\gdest$ is admissible in the following class $\mathcal C_1$, namely the class consisting of estimators 
\begin{equation}\label{est.in.C1}
\est = \bar X_1 + D\cdot \hat\phi\left(S^2_1, S^2_2\right)
\end{equation}
where $\hat\phi\colon (\reals_{>0})^2 \to [0,1]$ is an arbitrary function.
We shall presently show this is indeed the case. %

We start our discussion by reviewing some facts not involving any nonstandard analysis. 
Clearly, $\gdest$ is obtained by choosing 
\[
\hat\phi_{\text{GD}}(S^2_1,S^2_2) = \frac{S^2_1}{S^2_1+S^2_2}
\]
for $\hat\phi$ in \eqref{est.in.C1} and is thus seen to arise from a plug-in estimate for $\theta'$ based on estimating $\sigma^2_i$ by $S^2_i$.

Using see \cite[Lemma~2.1, p.~1605]{sinha1982} the risk of any estimator $\est \in \mathcal C_1$ given as in \eqref{est.in.C1}, for fixed $\theta = (\mu,\sigma^2_1, \sigma^2_2)$, is found to be
\begin{multline}\label{e.risk.sinha}
\riskf(\theta,\est) = \bias(\est)^2 + \var(\est) = \\
\underbrace{\expectation\left[D\cdot\hat \phi(S^2_1, S^2_2)\right]^2}_{=0}
+ 
\frac{\sigma_1^2\sigma_2^2}{n(\sigma_1^2+\sigma_2^2)} + 
\frac{\sigma_1^2 + \sigma_2^2}{n}
\cdot
\expectation \left[\hat\phi(S^2_1, S^2_2) - 
\frac{\sigma_1^2}{\sigma_1^2+\sigma_2^2}\right]^2
\end{multline}
where $\expectation\left[D\cdot\hat \phi\right]=0$ because $D$, $S^2_1$, and $S^2_2$ are independent and 
$\expectation(D)=0$.
Therefore  given %
a pair of estimators $\est_i = \bar X_1 + D\cdot \hat\phi_i\left(S^2_1, S^2_2\right)$ from $\mathcal C_1$, where $i\in\{0,1\}$, 
\begin{multline}\label{e.risk.diff1}
\riskf(\theta,\est_0) - \riskf(\theta,\est_1) =  \\
\expectation \left[
\frac{\sigma_1^2 + \sigma_2^2}{n}
\cdot 
\left(\hat\phi_0(S^2_1, S^2_2) - 
\frac{\sigma_1^2}{\sigma_1^2+\sigma_2^2}\right)^2 
-
\left(\hat\phi_1(S^2_1, S^2_2) - 
\frac{\sigma_1^2}{\sigma_1^2+\sigma_2^2}\right)^2
\right]
\end{multline}
In other words, as was already observed in \cite{sinha1982}, the relative risk of $\est_0, \est_1$ from $\mathcal C_1$ is determined by the relative performance of $\phi_0$ and $\phi_1$ in the related problem of estimating $\theta' :=\frac{\sigma_1^2}{\sigma_1^2+\sigma_2^2}$
from the data $S^2_1$ and $S^2_2$ with respect to the loss function given by
\[
\lossf'(\theta,\phi)= \frac{\sigma_1^2 + \sigma_2^2}{n} (\phi - \theta')^2.
\]
 We turn our attention to this latter problem.
Fix a positive (standard) real number $\alpha < \frac12$. 

Now bringing to bear ideas involving the use of infinitesimals, fix a positive infinitesimal $\nsparam$ and 
consider the hyperprior $\Pi$ for $(\sigma^2_1,\sigma^2_2)$ given by the (image under the star map of the) independent product of two inverse gamma distributions with shape $\alpha$ and scale $\nsparam$:
\begin{equation}\label{e.dist.sigma1}
\begin{gathered}
(\sigma^2_1,\sigma^2_2) \sim \invgamma(\alpha,\nsparam) \otimes \invgamma(\alpha,\nsparam), \\
\gammadist(x; \alpha,\nsparam) \propto \left(\frac1x\right)^{\alpha+1} \exp\left(\frac{\nsparam}x\right) 
\end{gathered}
\end{equation}
Since risk functions are continuous %
in the decision problem under consideration, the collection %
\[
\mathcal B := \{ O \subseteq \reals^2  \setdef \text{ $O$ open, $\bar O \subseteq(\reals_{>0})^2$}, \lambda(O)<\infty \} 
\]
is a determining family.

\begin{claim}
For any $O \in \mathcal B$ there is $C \in \reals$ such that $\Pi(\starmap O) > C \cdot \nsparam^{2\alpha}$.
\end{claim}
\begin{proof}
Let $O \in \mathcal B$ be given. 
Since by choice of $\mathcal B$,
\[
\varepsilon := \inf(\pr_0[O]\cup \pr_1[O])
\]
is a non-infinitesimal positive number and since $\nsparam \approx 0$, it holds for each $(\sigma^2_1,\sigma^2_2) \in O$ and $i\in\{0,1\}$ that $1 \approx \exp(\frac{\nsparam}{\sigma^2_i}) > 0.5$ and 
\begin{gather*}
\pi(O) = \iint_{\starmap O} \prod_{i=1,2}  \frac{{\nsparam}^\alpha}{\Gamma(\alpha)} e^{-\frac{\lambda_i}{2\sigma_i^2}} \left(\frac{1}{\sigma_i^2}\right)^{\alpha + 1} \;\mathrm{d}\sigma^2_i
\\
> \iint_{\starmap O} \prod_{i=1,2}  \frac{{\nsparam}^\alpha}{\Gamma(\alpha)} \frac12  \varepsilon^{\alpha + 1} \;\mathrm{d}\sigma^2_i =
\underbrace{\frac{\varepsilon^{2(\alpha+1)} \lambda(O) }{4\Gamma(\alpha)^2}}_{=C} \cdot \nsparam^{2\alpha} 
\end{gather*}
proving the lemma with $C$ chosen as indicated above.
\end{proof}
Let 
\begin{equation*}
\hat \Phi(S^2_1, S^2_2) = \frac{S^2_1 + \frac{2}{n-1}\nsparam}{S^2_1+S^2_2 + \frac{4}{n-1}\nsparam}.
\end{equation*}
and write $\Delta_\Pi$ for the hyperprocedure determined by setting $\hat\phi = \hat\Phi$ in \eqref{est.in.C1}. 
It is not hard to see that $\Delta_\Pi$ has minimal $\starmap{}$Bayes risk under $\Pi$ among $\starmap{\mathcal C}_1$ whenever $n>1$ (this was also observed in \cite{sinha1982}).

\begin{claim}
There exist $C' \in \reals_{>0}$ such that excess Bayes risk of $\starmap\gdest$ under $\Pi$ is at most $C'\cdot \nsparam$.
\end{claim}
\begin{proof}
We see from \eqref{e.risk.diff1} by straightforward calculation and exchanging the order of integration that
\begin{multline}\label{e.risk.diff2}
\riskf(\theta,\est_0) - \riskf(\theta,\est_1) =  \\
\expectation_{S^2_1, S^2_2 \sim \bar \Pi} \expectation_{\sigma^2_1, \sigma^2_2 \sim \Pi'} \left[
\frac{\sigma_1^2 + \sigma_2^2}{n}
\cdot 
\left(\hat\phi_0(\vec S)^2 - \hat\phi_1(\vec S)^2 \right) 
-
\frac{2\sigma^2_1}{n} \left(\hat\phi_0(\vec S) - \hat\phi_1(\vec S)\right)
\right]
\end{multline}
where we write $\phi_i(\vec S)$ as a short-hand for $\phi_i(S^2_1, S^2_2)$, and where the inner expectation is taken conditional on $S^2_1, S^2_2$ with $\sigma^2_1, \sigma^2_2$ treated as random variables following the posterior distribution here denoted by $\Pi'$, while the outer expectation is taken over $S^2_1, S^2_2$ which follow the
predictive prior distribution, which we denote by $\bar \Pi$.
By \eqref{e.dist.sigma1} the posterior distribution $\Pi'$ is a product of a pair of inverse gamma distributions: 
\begin{equation*}
(\sigma^2_1,\sigma^2_2) \sim \Pi'_1 \otimes \Pi'_2,
\end{equation*}
where for $i\in\{1,2\}$,
\begin{equation*}
\Pi'_i(x; \alpha',\nsparam'_i) \propto \left(\frac1x\right)^{\alpha'+1} \exp\left(\frac{\nsparam'_i}x\right)
\end{equation*}
with shape
\[
\alpha' = \alpha + \frac{n-1}2
\]
and scale
\[
\nsparam'_i = \nsparam + \frac{S^2_i}2
\]
and therefore with mean
\[
\expectation_{\sigma^2_i \sim \Pi'} \sigma^2_i = \frac{\nsparam'_i}{\alpha'-1} = \frac{\lambda + \frac{S^2_i}2}{\alpha + \frac{n-1}2-1}.
\]
From this and \eqref{e.risk.diff2}, a somewhat lengthy but straightforward calculation shows:
\begin{multline}\label{e.risk.diff3}
\riskf(\theta,\est_0) - \riskf(\theta,\est_1) =  
\expectation_{S^2_1, S^2_2 \sim \bar \Pi} 
\frac{4\nsparam (S^2_1 - S^2_2)^2}
{n(2\alpha + n -3)(S^2_1 + S^2_2)^2 (4\nsparam + S^2_2 + S^2_2)}
\leq \\
\expectation_{S^2_1, S^2_2 \sim \bar \Pi}  \frac{4\nsparam^2}{4\nsparam + S^2_2 + S^2_2}
\end{multline}
The predictive prior distribution $\bar\Pi$ on the other hand is an independent product of a pair of compound gamma distributions:
\begin{gather*}
(S^2_1, S^2_2) \sim \bar \Pi_1 \otimes \bar \Pi_2,\\
\bar \Pi_i(x;\alpha_0,\alpha_1,\tilde\beta_i) 
= 
 \int_{\reals_{>0}} \gammadist(x;\alpha_0,\beta)\gammadist(\mathrm{d}\beta; \alpha_1,\tilde\beta_i)
\\
\propto 
\int_{\reals_{>0}} \beta^{\alpha_0} e^{-\beta x} x^{\alpha_0 -1} e^{- \tilde\beta_i \beta} \beta^{\alpha_1 -1} \;\mathrm{d}\beta
\end{gather*}
with $\alpha_0 = \frac{n-1}{2}$, $\alpha_1 = \alpha$ and $\nsparam_i = \nsparam$ for each $i\in\{1,2\}$. 

Introduce new random variables $U_1, U_2$ by
\[
U_i := \frac{2\nsparam}{2\nsparam + S^2_i}
\]
for $i\in\{1,2\}$.
Then, as is easily verified,
\[
\frac{U_1}{1 + \frac{U_1}{U_2}}  = 
\frac{U_1 U_2}{U_1 + U_2} = \frac{4\nsparam}
{2\nsparam + S^2_2 + S^2_2}
\]
and moreover, $(U_1,U_2)$ are i.i.d.\ beta distributed (see \cite[p.~28,~Item~i)]{dubey}).
Thus we can continue \eqref{e.risk.diff3} as follows:
\begin{equation*}
\riskf(\theta,\est_0) - \riskf(\theta,\est_1) 
=  
\expectation_{S^2_1, S^2_2 \sim \bar \Pi}  \frac{4\nsparam^2}{4\nsparam + S^2_2 + S^2_2} 
= %
2\nsparam \cdot \expectation \frac{U_1}{1 + \frac{U_1}{U_2}} \leq 2 \nsparam
\end{equation*}
which proves the claim.
\end{proof}

\begin{corollary}\label{c.GD}
The Graybill-Deal estimator $\gdest$ is admissible among $\mathcal C_1$.
\end{corollary}
\begin{proof}
By the previous claims, firstly $\starmap\riskf(\Pi,\starmap{\gdest}) - \min_{\Delta \in \sigmamap\decisions} \starmap\riskf(\Pi,\Delta) \leq C'\cdot \nsparam$ , where $C' \in \reals_{>0}$ and secondly, for any $B\in\mathcal B$ there is $C\in \reals_{>0}$ such that $\Pi(B) > C\cdot \nsparam^{2\alpha}$.
The claim therefore follows by (\ref{i.blyth.B}) $\Rightarrow$ (\ref{i.adm}) in Theorem~\ref{t.adm}, as $\delta_0$ satisfies the nonstandard Blyth condition for $\mathcal B$:
\[
\frac{\starmap\riskf(\Pi,\starmap{\gdest}) - \min_{\Delta \in \sigmamap\decisions}\starmap\riskf(\Pi,\Delta)} {\nsparam^{2\alpha}} 
\leq
\frac{\starmap\riskf(\Pi,\starmap{\gdest}) - \starmap\riskf(\Pi,\Delta_\Pi)} {\nsparam^{2\alpha}} 
\leq 
\frac{C}{C'}\cdot \frac{\nsparam}{\nsparam^{2\alpha}} \approx 0
\]
where the last relation holds by choice of $\alpha$ since $1- 2 \alpha\in\reals^{>0}$ and so $\nsparam^{1-2\alpha} \approx 0$.
\end{proof}

\subsection{Finitely additive priors}

For the following application, 
we ask the reader to recall (from \cite[Section~7]{duanmuroy} or from \cite{duanmu-weiss}) the notion of the internal push-down of an internal measure.
Suppose $\Pi$ is an internal probability measure---or even just an internal probability charge, or as it is sometimes called, a finitely additive probability measure---on $\starmap\parametersp$.
Suppose further that $\mathcal F$ is a $\sigma$-algebra on $\parametersp$ such that
$\starmap{\mathcal F} \subseteq \dom(\Pi)$.
Then we can define a positive probability charge (finitely additive probability measure)
\[
\Pi^p \colon \mathcal F \to [0,1]
\]
by letting, for $X \in \mathcal F$,
\[
\Pi^p(X) = \ST\Pi(\starmap X)
\]
The probability charge $\Pi^p$ is called the \emph{internal push-down of $\Pi$}.

Using the above lemma, we can characterize admissibility using finitely additive priors, in problems where the set of decision procedures is compact and risk functions are bounded.
Note that these assumptions do not entail that the risk is totally bounded.

The following crucial lemma establishes a useful connection between integration with respect to $\Pi$ and with respect to $\Pi^p$.

\begin{lemma}[see 5.4 in \cite{duanmu-weiss}]
Suppose $(\parametersp, \mathcal F)$ is a measurable space, $f \colon \parametersp \to \reals$ is bounded and measurable, 
and $\Pi$ is an internal measure on $(\starmap\parametersp,\starmap{\mathcal F})$.
Then $\int_{\starmap \parametersp} \starmap f \;\mathrm{d}\Pi \approx  \int_{\parametersp} f \;\mathrm{d}\Pi^p$.
\end{lemma}

\begin{theorem}\label{t.bounded.risk}
Suppose that $\decisionproblem$ is a decision problem such that  
\begin{enumerate}
\item For any $\delta \in \decisions$, the risk function $\riskf^\delta$ is bounded. 
\item There exists a topology with respect to which $\decisions$ is compact and 
$\riskf(\theta,\cdot)$ is continuous for each $\theta\in\parametersp$.
\end{enumerate}
Then the following are equivalent:
\begin{enumerate}
\item $\delta_0$ is admissible among $\decisions$.
\item $\delta_0$ satisfies the finitely additive Stein's condition.
\end{enumerate}
\end{theorem}
\begin{proof}
We show (1)$\Rightarrow$(2) (in the other direction, the usual proof goes through).

Let a problem and $\delta_0$ be as in the theorem and also $\theta_0\in \parametersp$ and $\varepsilon \in \reals_{>0}$ be given.
Let
\begin{align*}
U_0 &:= \{\delta \in \decisions \setdef \riskf(\theta_0,\delta) > \riskf(\theta_0,\delta_0) - \varepsilon\},\\
U_{\varepsilon', \theta} &:= \{\delta \in \decisions \setdef \riskf(\theta,\delta) > \riskf(\theta,\delta_0) + \varepsilon'\}
\end{align*}
for $\theta\in \parametersp$ and $\varepsilon' \in \reals^{>0}$.
Since $\delta_0$ is admissible in $\decisions$, for any $\delta \in \decisions$ if $\delta \notin U_0$, there is $\theta \in \parametersp$ and $\varepsilon' \in \reals^{>0}$ such that $\delta \in U_{\varepsilon', \theta}$. 
That is, $\{U_0\}\cup\{U_{\varepsilon', \theta}\setdef \theta \in \parametersp, \varepsilon' \in \reals^{>0} \}$ covers $\decisions$ and therefore we can find a finite subcover 
\[
U_0 ,U_{\varepsilon_1, \theta_1}, \hdots U_{\varepsilon_n, \theta_n}.
\]
For each $\delta \in \decisions\setminus U_0$ by construction there exists $i$ such that $r_{\delta_0}(\theta_i,\delta)> \varepsilon_i$. 
We can find $\gamma \in \reals^{\geq 0}$ such that for all $\delta \in \decisions\setminus U_0$,
\begin{equation}\label{e.gamma}
\big(\exists i \in\{1,\hdots,n\}\big)\; \left[ r_{\delta_0}(\theta_i,\delta)\neq 0 \land \gamma \geq - \frac{r_{\delta_0}(\theta_0,\delta)+\varepsilon}{r_{\delta_0}(\theta_i,\delta)} \right]
\end{equation}
Namely let 
\begin{multline*}
\gamma := 
\sup_{\delta\in\decisions\setminus U_0} \min \left\{ 
- 
\frac{r_{\delta_0}(\theta_0,\delta)+\varepsilon}
{r_{\delta_0}(\theta_i,\delta)}  \setdef i \in\{1,\hdots,n\}, r_{\delta_0}(\theta_i,\delta) \neq 0 \right\}
=\frac{-\left(\min_{\delta \in\decisions} r_{\delta_0}(\theta_0,\delta) + \varepsilon\right)}
{\min \big\{\varepsilon_i \setdef i \in\{1,\hdots,n\}\big\}}
\end{multline*}
where the second equality holds because $\delta \in U_{\varepsilon_1, \theta_1}\cup \hdots \cup U_{\varepsilon_n, \theta_n}$ and because $\delta\notin U_0$ entails $r_{\delta_0}(\theta_0,\delta) +\varepsilon\leq 0$, whence the numerator of the last fraction is non-negative.

Observe that also for $\delta_0 \in U_0$, \eqref{e.gamma} obtains as well: In this case, as $r_{\delta_0}(\theta_0,\delta)+\varepsilon > 0$, the entire expression on the
right of the inequality in \eqref{e.gamma} will be negative if  $r_{\delta_0}(\theta_i,\delta)>0$, and therefore \eqref{e.gamma} obtains since $\gamma \geq 0$. 

Since $\delta_0$ is admissible in $\decisions$, by saturation and the proof of Theorem~\ref{t.hyperprior.admissibility} we can find hyperfinite sets $\tilde \parametersp \subseteq \starmap\parametersp$ and $\tilde\decisions\subseteq \starmap\decisions$ such that $\{\starmap\theta_1, \hdots, \starmap\theta_n\} \subseteq \tilde\parametersp$ and 
$\sigmamap U_0 \subseteq \tilde \decisions$ and for every $\Delta \in \tilde \decisions$ there exists $\tilde \theta \in \tilde \parametersp$ such that 
$\starmap{r}_{\delta_0}(\tilde\theta,\Delta)>0$.

By the transfer principle \eqref{e.gamma} holds for every $\Delta \in \starmap{\decisions}$. 
In particular, we have:
\begin{equation}\label{e.value}
(\forall \Delta \in \tilde \decisions)(\exists \tilde \theta \in \tilde\parametersp) \; \starmap\riskf_{\starmap\delta_0}(\starmap\theta_0,\Delta) + \gamma \cdot \starmap\riskf_{\starmap\delta_0}(\tilde\theta,\Delta) \geq -\varepsilon
\end{equation}
Consider the internal decision problem $\tilde\aproblem=(\tilde\parametersp, \sconvexhull{\tilde\decisions}, \starmap{\actionsp}, \tilde\riskf, \bar\priors)$ where the convex risk function is given by 
\[
\tilde\riskf(\tilde\theta,\Delta) := \starmap\riskf_{\starmap\delta_0}(\starmap\theta_0,\Delta) + \gamma \cdot \starmap\riskf_{\starmap\delta_0}(\tilde\theta,\Delta),
\]
for $\tilde\theta\in \tilde\parametersp$ and $\Delta \in \tilde\decisions$, 
and where $\tilde\priors$ is the set of internal (hyperfinite) probability measures on $\tilde\parametersp$.
Note that $\tilde\decisions$ and $\tilde\parametersp$ are hyperfinite, and moreover therefore $\tilde\riskf$ is bounded. 
By \eqref{e.value} and the transfer of \cite[Theorem~2.1, p.~37]{wald1971} the internal decision problem $\tilde\aproblem$ gives rise to a determined game with value at least $-\varepsilon$ and so we may find a $\Pi\in\tilde\priors$
 such that 
\begin{equation}\label{e.pre.ns.Stein}
(\forall \Delta \in \tilde \decisions)\: \starmap\riskf_{\starmap\delta_0}(\theta_0,\Delta) + \gamma \cdot \starmap\riskf_{\starmap\delta_0}(\Pi,\Delta) \geq -\frac{3\varepsilon}{2}
\end{equation}
By assumption, $\Pi \in \starmap \priors$.
Letting
\[
\bar\Pi := \frac{1}{\gamma+1} \starmap{\dirac{\theta_0}} + \frac{\gamma}{\gamma+1} \Pi,
\]
and dividing by $1+\gamma$ in both sides of the inequality in \eqref{e.pre.ns.Stein} and restricting the first quantifier to $\sigmamap\decisions$, 
it follows that
\[
(\forall \Delta \in \sigmamap\decisions)\; \starmap\riskf_{\starmap\delta_0}(\bar\Pi,\Delta)  \geq -\frac{1}{\gamma+1}\frac{3\varepsilon}{2}
\]
and moreover, the same holds for (the star image of) the internal push-down measure $({\bar\Pi})^p$. 
Since ${\bar\Pi}^p(\{\theta_0\}) = \ST\big({\bar\Pi}(\{\theta_0\})) \geq \frac{1}{1+\gamma}$, we have 
\[
(\forall \delta \in \decisions)\; \riskf_{\delta_0}({\bar\Pi}^p,\delta)  \geq -{\bar\Pi}^p(\{\delta_0\})\cdot \frac{3}{2} \cdot \varepsilon
\]
or equivalently,
\[
\riskf({\bar\Pi}^p,\delta_0)  - \inf_{\delta\in\decisions}\riskf({\bar\Pi}^p,\delta)  \leq {\bar\Pi}^p(\{\delta_0\})\cdot \frac{3}{2} \cdot \varepsilon
\]
which (the additional factor of $\frac32$ being irrelevant) proves the theorem.
\end{proof}

\bibliography{admissibility}{}
\bibliographystyle{amsplain}

\vfill

\end{document}

%% file: preamble.tex
\theoremstyle{plain}
\newtheorem{theorem}{Theorem}[section]
\newtheorem{lemma}[theorem]{Lemma}

\newtheorem{corollary}[theorem]{Corollary}
\newtheorem{claim}[theorem]{Claim}
\newtheorem{question}[theorem]{Question}

\newtheorem*{claim*}{Claim}
\newtheorem*{theorem*}{Theorem}
\newtheorem*{lemma*}{Lemma}
\newtheorem*{proposition*}{Proposition}
\newtheorem*{corollary*}{Corollary}
\newtheorem*{subclaim*}{Subclaim}

\theoremstyle{definition}
\newtheorem{definition}[theorem]{Definition}

\newtheorem*{definition*}{Definition}
\newtheorem*{example*}{Example}
\newtheorem*{fact*}{Fact} 

\theoremstyle{remark}
\newtheorem{remark}[theorem]{Remark}

{\end{minipage}\end{equation}}

{\end{minipage}\end{equation}}

\newenvironment{eqpar*}{\begin{equation*}\begin{minipage}{0.8\columnwidth}}%
{\end{minipage}\end{equation*}}

\newcommand{\nsparam}{\tilde\beta}

\DeclareMathOperator{\pr}{pr}

\DeclareMathOperator{\est}{\hat\mu}

\DeclareMathOperator{\normaldist}{\mathcal N}
\DeclareMathOperator{\invgamma}{inv-gamma}
\DeclareMathOperator{\gammadist}{gamma}
\DeclareMathOperator{\gdest}{\hat{\mu}_{\text GD }}
\DeclareMathOperator{\expectation}{\mathbb{E}}
\DeclareMathOperator{\bias}{Bias}
\DeclareMathOperator{\var}{Var}

\newcommand{\dirac}[1]{\operatorname{Dirac}_{#1}}

\DeclareMathOperator{\prob}{\mathcal M}

\newcommand{\expected}[1]{\operatorname{\mathbb{E}}_{#1}}

\newcommand{\parametersp}{\Theta}
\newcommand{\samplesp}{\mathbb X}
\newcommand{\actionsp}{\mathbb A}
\newcommand{\lossf}{\ell}
\newcommand{\model}{P}
\newcommand{\decisions}{\mathcal D}
\newcommand{\priors}{\mathcal R}
\newcommand{\riskf}{r}
\newcommand{\aproblem}{e}

\newcommand{\sconvexhull}[1]{{#1}^{(\starmap{}\text{C.H.})}}

\newcommand{\decisionproblem}{(\parametersp,\actionsp,\lossf,\samplesp,\model,\decisions,\priors)}

\newcommand{\starmap}[1]{{}^*{#1}}
\newcommand{\sigmamap}[1]{{}^\sigma{#1}}
\newcommand{\ST}{{}^{\circ}}

\newcommand{\bsim}[1]{\mathbin{\sim_{#1}}}
\newcommand{\nbsim}[1]{\mathbin{\not\sim_{#1}}}

\newcommand{\rsim}[1]{\bsim[r]}
\newcommand{\nrsim}[1]{\nbsim[r]}

\newcommand{\NS}[1]{\operatorname{NS}({#1})}

\DeclareMathOperator{\powerset}{\mathcal{P}}

\DeclareMathOperator{\nat}{\mathbb{N}}

\providecommand{\int}{\mathbb{Z}}
\providecommand{\reals}{\mathbb{R}}

\DeclareMathOperator{\dom}{dom}

\providecommand{\setdef}{\;|\;}

\providecommand{\supp}{\mathrm{ supp }}